\documentclass[11pt]{article}
\usepackage{etex}
\reserveinserts{28}
\usepackage[all]{xy}

\usepackage{amsfonts}
\usepackage{amsthm}
\usepackage{enumerate}
\usepackage{graphicx}
\usepackage{mathrsfs}
\usepackage{bm}
\usepackage{cite}
\usepackage[colorlinks,linkcolor=blue,anchorcolor=blue,citecolor=blue]{hyperref}
\usepackage{amssymb,amsmath} % font used for R in Real numbers
\usepackage{indentfirst}
\usepackage{makecell}
\usepackage{tikz}
\usepackage{pgf}
\usetikzlibrary{patterns}
\usepackage{pgffor}
\usepackage{pgfcalendar}
\usepackage{pgfpages}
\usepackage{shuffle,yfonts}
\usepackage{mathtools}
\DeclareFontFamily{U}{shuffle}{}
\DeclareFontShape{U}{shuffle}{m}{n}{ <-8>shuffle7 <8->shuffle10}{}

\newcommand{\A}{{\rm A}}

\newcommand\ta{{\texttt{a}}}
\newcommand\tb{{\texttt{b}}}

\newcommand{\bfj}{{\boldsymbol{\sl{j}}}}
\newcommand{\bfk}{{\boldsymbol{\sl{k}}}}

\newcommand{\bfl}{{\boldsymbol{\sl{l}}}}

 \allowdisplaybreaks
\usetikzlibrary{arrows,shapes,chains}

\catcode`!=11
\let\!int\int \def\int{\displaystyle\!int}
\let\!lim\lim \def\lim{\displaystyle\!lim}
\let\!sum\sum \def\sum{\displaystyle\!sum}
\let\!sup\sup \def\sup{\displaystyle\!sup}
\let\!inf\inf \def\inf{\displaystyle\!inf}
\let\!cap\cap \def\cap{\displaystyle\!cap}
\let\!max\max \def\max{\displaystyle\!max}
\let\!min\min \def\min{\displaystyle\!min}
\catcode`!=12

\let\oldsection\section
\renewcommand\section{\setcounter{equation}{0}\oldsection}

\allowdisplaybreaks

\DeclareMathOperator*{\dep}{dep}
\DeclareMathOperator{\Li}{Li}

\def\N{\mathbb{N}}

\def\CC{\mathbb{C}}

\def\ze{\zeta}

\theoremstyle{plain}
\newtheorem{thm}{Theorem}[section]

\newtheorem{cor}[thm]{Corollary}

\theoremstyle{definition}

\begin{document}
%%%%%%%%%%%%%%%%%%%% title %%%%%%%%%%%%%%%%%%%%%%%%%%%%%%%%%%%%%%%%%%%%%%%%
\title{\bf New Proofs of the Explicit Formulas of Arakawa--Kaneko Zeta Values and Kaneko--Tsumura $\eta$- and $\psi$- Values}
\author{
{Masanobu Kaneko${}^{a,}$\thanks{Email: mkaneko@math.kyushu-u.ac.jp (M. Kaneko), ORCID 0000-0002-6658-7313},\ Weiping Wang${}^{b,}$\thanks{Email: wpingwang@zstu.edu.cn (W. Wang), ORCID 0009-0001-5162-0598}, \ Ce Xu${}^{c,}$\thanks{Email: cexu2020@ahnu.edu.cn (C. Xu),\ {\bf corresponding author}, ORCID 0000-0002-0059-7420}\ and\ Jianqiang Zhao${}^{d,}$\thanks{Email: zhaoj@ihes.fr (J. Zhao), ORCID 0000-0003-1407-4230}}\\[1mm]
\small a. Faculty of Mathematics, Kyushu University
744 Motooka, Nishi-ku, Fukuoka, 819-0395, Japan\\
\small b. School of Science, Zhejiang Sci-Tech University, Hangzhou 310018, P.R. China\\
\small c. School of Mathematics and Statistics, Anhui Normal University, Wuhu 241002, P.R. China\\
\small d. Department of Mathematics, The Bishop's School, La Jolla, CA 92037, United States of America
}

\date{}
\maketitle

\noindent{\bf Abstract.} In this paper, we establish some new identities of integrals involving multiple polylogarithm functions and their level two analogues in terms of Hurwitz-type multiple zeta (star) values. Using these identities, we provide new proofs of the explicit formulas of Arakawa--Kaneko zeta values, Kaneko--Tsumura $\eta$- and $\psi$-values, and also give a formula for double $T$-values.



\section{Introduction}

We begin with some basic notations. Let $\N$ be the set of positive integers and $\N_0:=\N\cup\{0\}$.
A finite sequence $\bfk:=(k_1,\ldots, k_r)\in\N^r$ is called a \emph{composition}. We put $|\bfk|:=k_1+\cdots+k_r$ and $\dep(\bfk):=r$, and call them the \emph{weight} and the \emph{depth} of $\bfk$, respectively. If $k_1>1$, $\bfk$ is called \emph{admissible}.
For a composition $\bfk=(k_1,k_2,\ldots,k_r)$ and $\bfj:=(j_1,j_2,\ldots,j_r)\in \N_0^r$, let $\overleftarrow{\bfk}:=(k_r,\ldots,k_2,k_1)$ be the \emph{reversal} of $\bfk$ and
\[
B(\bfk;\bfj):=\prod_{i=1}^r\binom{k_i+j_i-1}{j_i}.
\]
Moreover, for such $\bfk$ and $\bfj$ with same depth, we denote by $\bfk+\bfj$ the composition obtained by adding the corresponding components.

The \emph{Hoffman dual} of a composition $\bfk=(k_1,\ldots,k_r)$ is the composition $\bfk^\vee=(k'_1,\ldots,k'_{r'})$ uniquely
determined by the conditions
$|\bfk|=k_1+\cdots+k_r=k'_1+\cdots+k'_{r'}$ and
\begin{equation*}
\{1,2,\ldots,|\bfk|-1\}
=\left\{\sum_{i=1}^{j} k_i\right\}_{j=1}^{r-1}
 \coprod \left\{\sum_{i=1}^{j} k_i'\right\}_{j=1}^{r'-1},
\end{equation*}
where $\coprod$ denotes the union of two disjoint sets. Equivalently, $\bfk^\vee$ can be obtained from $\bfk$ by swapping the commas ``,'' and the plus signs ``+'' in the expression
\begin{equation}\label{eq:HdualDef}
 \bfk=(\underbrace{1+\cdots+1}_{\text{$k_1$ times}},\dotsc,\underbrace{1+\cdots+1}_{\text{$k_r$ times}}),
\end{equation}
so that
\begin{align}\label{eq:HdualDef2}
{\bfk}^\vee=(\underbrace{1,\ldots,1}_{k_1}+\underbrace{1,\ldots,1}_{k_2}+1,\ldots,1+\underbrace{1,\ldots,1}_{k_r}).
\end{align}
For example, we have
$({1,1,2,1})^\vee=(3,2)$ and $({2,1,4})^\vee=(1,3,1,1,1)$.

For an admissible composition written as $\bfk_+:=(k_1+1,k_2,\ldots,k_r)$ with $\bfk=(k_1,k_2,\ldots,k_r)$,
the usual dual composition of $\bfk_+$ in the theory of multiple zeta values (see for instance \cite[Ch.~5]{Zhao2016} for the precise definition)
can be described in terms of Hoffman's dual $\vee$ as $(\overleftarrow{\bfk}^\vee)_+$.

Now, for an admissible composition $\bfk=(k_1,\ldots,k_r)$, the \emph{multiple zeta values} (MZVs) and \emph{multiple zeta star values} are defined by
\begin{align*}
&\zeta(\bfk):=\sum_{n_1>\cdots>n_r>0}
    \frac{1}{n_1^{k_1}\cdots n_r^{k_r}}\quad\text{and}\quad
\zeta^\star(\bfk):=\sum_{n_1\geq\cdots\geq n_r>0}
    \frac{1}{n_1^{k_1}\cdots n_r^{k_r}},
\end{align*}
respectively. The systematic study of MZVs began with the works of Hoffman \cite{H1992} and Zagier \cite{DZ1994} in the early 1990s. Due to their surprising and sometimes mysterious appearance in the study of many branches of mathematics and theoretical physics, these special values have attracted a lot of attention and interest in the past three decades. For more details, the readers are referred to the book of the fourth-named author \cite{Zhao2016}.

As a level two generalization of MZVs, the first-named author and Tsumura \cite{KanekoTs2018b,KanekoTs2019} introduced and studied \emph{multiple $T$-values}, defined by
\begin{align*}
T(\bfk):
&=\sum_{\substack{n_1>n_2>\cdots>n_r>0 \\ n_j\equiv r+1-j\pmod{2}}}
    \frac{2^r}{n_1^{k_1}n_2^{k_2}\dotsm n_r^{k_r}}\nonumber\\
&=\sum_{m_1>m_2>\cdots>m_r>0}{\frac{2^r}{{(2m_1-r)^{{k_1}}(2m_2-r+1)^{{k_2}} \dotsm (2m_r-1)^{{k_r}}}}}.
\end{align*}

On the other hand, for any composition $\bfk=(k_1,k_2\ldots,k_r)$ and $s\in\CC$ with $\Re(s)>0$, the \emph{Arakawa-Kaneko zeta function} \cite{AM1999} is defined by
\begin{align}\label{a1}
\xi(s;\bfk):=\frac{1}{\Gamma(s)}\int_{0}^\infty
    \frac{t^{s-1}}{e^t-1}{\rm Li}_{\bfk}(1-e^{-t})dt,
\end{align}
where
\begin{align}\label{a2}
&{\mathrm{Li}}_{\bfk}(z)
:=\sum_{n_1>\cdots>n_r>0}{\frac{{{z^{{n_1}}}}}{{n_1^{k_1}\cdots n_r^{{k_r}}}}}\,,\quad z \in [-1,1)
\end{align}
is the \emph{multiple polylogarithm function}.
In \cite{KT2018}, the first-named author and Tsumura introduced and studied a function of similar type
\begin{align}\label{Def-eta}
\eta(s;\bfk):=\frac{1}{\Gamma(s)} \int_{0}^\infty \frac{t^{s-1}}{1-e^t}{\rm Li}_{\bfk}(1-e^{t})dt,
\end{align}
which we call the \emph{Kaneko-Tsumura $\eta$-function}.

They also defined in \cite{KanekoTs2018b} a level two analogue of $\xi(s;\bfk)$, which is referred to as the
\emph{Kaneko-Tsumura $\psi$-function}, as
\begin{equation}\label{defn-formula-KTP}
\psi(s;\bfk)=\frac{1}{\Gamma(s)}\int_0^\infty t^{s-1}
    \frac{{\rm A}(\bfk;\tanh t/2)}{\sinh(t)} dt,
\end{equation}
for any composition $\bfk=(k_1,k_2\ldots,k_r)$ and $s\in\CC$ with $\Re{(s)}>0$. Here,
\begin{align}\label{defn-formula-KTA}
{\rm A}(\bfk;x):
&= 2^r\sum_{\substack{n_1>n_2>\cdots> n_r>0\\ n_j\equiv r+1-j\pmod{2}}} {\frac{{{x^{{n_1}}}}}{{n_1^{{k_1}}n_2^{{k_2}} \cdots n_r^{{k_r}}}}}\nonumber\\
&=2^r \sum\limits_{m_1>m_2>\cdots>m_r>0}  {\frac{{{x^{{2m_1-r}}}}}{{(2m_1-r)^{{k_1}}(2m_2-r+1)^{{k_2}} \dotsm (2m_r-1)^{{k_r}}}}}
\end{align}
for $|x|\leq 1$, with $(k_1,x)\neq (1,1)$ (note that this is $2^r$ times ${\rm Ath}(\bfk;z)$ introduced in \cite{KanekoTs2018b}).

Special values of these $\xi$-, $\eta$-, and $\psi$-functions have been studied in many works such as \cite{AM1999,ChenKW19,KT2018,KanekoTs2018b,LuoSi2023,PX2019,XuZhao2020a,XuZhao2020b,Y2016} (see also references therein).
In particular, the first-named author and Tsumura \cite{KT2018} (for $\xi$ and $\eta$) and the third- and the fourth-named author \cite{XuZhao2020a} (for $\psi$) proved the explicit formulas
\begin{align}
&\xi(k+1;\bfk)=\sum_{|\bfj|=k,\dep(\bfj)=n} B\left((\overleftarrow{\bfk}^\vee)_+;\bfj\right)
    \ze\left((\overleftarrow{\bfk}^\vee)_++\bfj\right),\\
&\psi(k+1;\bfk)=\sum_{|\bfj|=k,\dep(\bfj)=n} B\left((\overleftarrow{\bfk}^\vee)_+;\bfj\right)T\left((\overleftarrow{\bfk}^\vee)_++\bfj\right),\label{eq-ef-KTPsiF}\\
&\eta(k+1;\bfk)=(-1)^{r-1} \sum_{|\bfj|=k,\dep(\bfj)=n} B\left((\overleftarrow{\bfk}^\vee)_+;\bfj\right)\ze^\star\left((\overleftarrow{\bfk}^\vee)_++\bfj\right),\label{eq-ef-KTeta}
\end{align}
where $\bfk:=(k_1,k_2,\ldots,k_r)$, $k\in \N_0$ and $n:=|\bfk|+1-\dep(\bfk)$.

In this short note, we provide new proofs of these formulas (Corollaries \ref{AKZ-KTPSI-COR} and \ref{KTeta-COR}) as a result of the study of certain integrals involving multiple polylogarithm
functions and their level two analogue, the A-functions~\eqref{defn-formula-KTA}. We also prove a formula for double $\zeta$- and $T$-values related to a conjecture posed in \cite[Conjecture 5.3]{KanekoTs2019}.

\section{New Proofs of Arakawa--Kaneko zeta-type values}\label{sec-mt}

In this section, we will combine the iterated integrals with the integrals involving multiple polylogarithm functions and A-functions to prove some interesting results.
The theory of iterated integrals was developed first by K.T. Chen in the 1960's \cite{KTChen1971,KTChen1977}. It has played important roles in the study of algebraic topology and algebraic geometry in the past half century. Its simplest form is
$$\int_{a}^b f_p(t)dtf_{p-1}(t)dt\cdots f_1(t)dt:=\int_{a<t_p<\cdots<t_1<b}f_p(t_p)f_{p-1}(t_{p-1})\cdots f_1(t_1)dt_1dt_2\cdots dt_p.$$

Similarly to the multiple zeta values and multiple zeta star values, for a composition $\bfk=(k_1,\ldots,k_r)$, a positive integer $n$, and $\alpha\in \mathbb{C}\backslash\N_0^-$ with $\N_0^-:=\{0,-1,-2,-3,\ldots\}$, we define the \emph{Hurwitz-type multiple zeta values}, \emph{Hurwitz-type multiple zeta star values} and \emph{Hurwitz-type multiple $T$-values} by
\begin{align}
\zeta(\bfk;\alpha)&:=\sum\limits_{n_1>\cdots>n_r>0 } \frac{1}{(n_1+\alpha-1)^{k_1}\cdots (n_r+\alpha-1)^{k_r}}\,,\label{HTMZVs}\\
\zeta^\star(\bfk;\alpha)&:=\sum\limits_{n_1\geq\cdots\geq n_r>0} \frac{1}{(n_1+\alpha-1)^{k_1}\cdots (n_r+\alpha-1)^{k_r}}\,,\label{THMZSVs}
\end{align}
and
\begin{equation}\label{Defn-HTMTVs}
T(\bfk;\alpha):=\sum\limits_{m_1>\cdots>m_r>0}  {\frac{2^r}{{(2m_1-r-1+\alpha)^{{k_1}}(2m_2-r+\alpha)^{{k_2}} \dotsm (2m_r-2+\alpha)^{{k_r}}}}},
\end{equation}
respectively. Clearly, $\zeta(\bfk;1)=\zeta(\bfk)$, $\zeta^\star(\bfk;1)=\zeta^\star(\bfk)$ and $T(\bfk;1)=T(\bfk)$. In particular, according to the definitions, for any admissible composition $\bfk:=(k_1,k_2,\ldots,k_r)$ and $\alpha \in \mathbb{C}\setminus \N^-$ with $\N^-:=\{-1,-2,-3,\ldots\}$, we have
\begin{align}
\Li_{\bfk}(x)
&=\int_0^x
    \left(\frac{dt}{1-t}\right)\left(\frac{dt}{t}\right)^{k_r-1}\cdots
    \left(\frac{dt}{1-t}\right)\left(\frac{dt}{t}\right)^{k_1-1},\label{Eq-MPL-ItIn}\\
\A(\bfk;x)
&=\int_0^x
    \left(\frac{2dt}{1-t^2}\right)\left(\frac{dt}{t}\right)^{k_r-1}\cdots
    \left(\frac{2dt}{1-t^2}\right)\left(\frac{dt}{t}\right)^{k_1-1},\label{Eq-KTA-ItIn}\\
\ze(\bfk;1+\alpha)
&=\int_0^1
    \frac{t^\alpha dt}{1-t}\left(\frac{dt}{t}\right)^{k_r-1}
        \left(\frac{dt}{1-t}\right)\left(\frac{dt}{t}\right)^{k_{r-1}-1}\cdots
        \left(\frac{dt}{1-t}\right)\left(\frac{dt}{t}\right)^{k_1-1},
        \label{HTMZVs-Iterated-Integeral}\\
T(\bfk;1+\alpha)
&=\int_0^1 \frac{2t^\alpha dt}{1-t^2}\left(\frac{dt}{t}\right)^{k_r-1}
    \left(\frac{2dt}{1-t^2}\right)\left(\frac{dt}{t}\right)^{k_{r-1}-1}\cdots
        \left(\frac{2dt}{1-t^2}\right)\left(\frac{dt}{t}\right)^{k_1-1}.
        \label{HTMTVs-Iterated-Integeral}
\end{align}
Then the following theorem can be established.

\begin{thm}\label{thm-MPL-MAL-Para}
For any composition $\bfk:=(k_1,k_2,\ldots,k_r)$, $k\in \N$, and any $\alpha\in \mathbb{C}\backslash \N$, the following formulas hold:
\begin{align}
&\int_0^1\frac{\Li_{\bfk}(x)\log^k(1-x)}{x(1-x)^\alpha}dx
=(-1)^kk!\sum_{\substack{|\bfj|=k\\\dep(\bfj)=n}}
    B\left((\overleftarrow{\bfk}^\vee)_+;\bfj\right)
    \ze\left((\overleftarrow{\bfk}^\vee)_++\bfj;1-\alpha\right),\label{EQ-MPL-LOG-A}\\
&\int_0^1\frac{\A(\bfk;x)\log^k(\frac{1-x}{1+x})}{x(\frac{1-x}{1+x})^\alpha}dx
=(-1)^kk!\sum_{\substack{|\bfj|=k\\\dep(\bfj)=n}}
    B\left((\overleftarrow{\bfk}^\vee)_+;\bfj\right)
    T\left((\overleftarrow{\bfk}^\vee)_++\bfj;1-\alpha\right),\label{EQ-MAL-LOG-A}
\end{align}
where $n:=|\bfk|+1-\dep(\bfk)$.
\end{thm}
\begin{proof}
We first consider the two integrals
\[\int_0^1 \frac{\Li_{\bfk}(x)}{x(1-x)^\alpha}dx\quad \text{and}\quad \int_0^1 \frac{\A(\bfk,x)}{x(\frac{1-x}{1+x})^\alpha}dx.\]
Using the iterated integral expressions we get
\begin{align}
&\int_0^1 \frac{\Li_{\bfk}(x)}{x(1-x)^\alpha}dx
    =\int_0^1\left(\frac{dt}{1-t}\right)\left(\frac{dt}{t}\right)^{k_r-1}\cdots
        \left(\frac{dt}{1-t}\right)\left(\frac{dt}{t}\right)^{k_1-1}
        \frac{dt}{t(1-t)^\alpha}\label{eq-MPL-IT-NOT-Log},\\
&\int_0^1 \frac{\A(\bfk;x)}{x(\frac{1-x}{1+x})^\alpha}dx
    =\int_0^1\left(\frac{2dt}{1-t^2}\right)\left(\frac{dt}{t}\right)^{k_{r}-1}\cdots
        \left(\frac{2dt}{1-t^2}\right)\left(\frac{dt}{t}\right)^{k_1-1}
        \frac{dt}{t(\frac{1-t}{1+t})^\alpha}.\label{eq-KTA-IT-NOT-Log}
\end{align}
Applying the change of variable $t\rightarrow 1-t$ in \eqref{eq-MPL-IT-NOT-Log} and $t\rightarrow \frac{1-t}{1+t}$ in \eqref{eq-KTA-IT-NOT-Log}, and using the definition of Hoffman's dual with the help of \eqref{HTMZVs-Iterated-Integeral} and \eqref{HTMTVs-Iterated-Integeral}, we obtain
\begin{align}
&\int_0^1 \frac{\Li_{\bfk}(x)}{x(1-x)^\alpha}dx
    =\ze \left((1,\overleftarrow{\bfk})^\vee;1-\alpha\right)
    =\ze \left((\overleftarrow{\bfk}^\vee)_+;1-\alpha\right),
        \label{eq-MPL-IT-NOT-Log-HTMZVs}\\
&\int_0^1 \frac{\A(\bfk;x)}{x(\frac{1-x}{1+x})^\alpha}dx
    =T\left((1,\overleftarrow{\bfk})^\vee;1-\alpha\right)
    =T\left( (\overleftarrow{\bfk}^\vee)_+;1-\alpha\right).
        \label{eq-KTA-IT-NOT-Log-HTMTVs}
\end{align}
Hence, differentiating \eqref{eq-MPL-IT-NOT-Log-HTMZVs} and \eqref{eq-KTA-IT-NOT-Log-HTMTVs} $k$ times with respect to $\alpha$, we have
\begin{align*}
&\int_0^1 \frac{\Li_{\bfk}(x)\log^k(1-x)}{x(1-x)^\alpha}dx=(-1)^k\frac{d^k}{d\alpha^k}\ze \left( (\overleftarrow{\bfk}^\vee)_+;1-\alpha\right),\\
&\int_0^1 \frac{\A(\bfk;x)\log^k(\frac{1-x}{1+x})}{x(\frac{1-x}{1+x})^\alpha}dx=(-1)^k\frac{d^k}{d\alpha^k}T\left( (\overleftarrow{\bfk}^\vee)_+;1-\alpha\right).
\end{align*}
Finally, using the \emph{Leibniz rule}
\begin{align*}
\left(\prod_{j=1}^p f_j\right)^{(k)}=\sum\limits_{\substack{k_1+k_2+\cdots+k_p=k\\ k_1,k_2,\ldots,k_p\in\N_0}} \frac{k!}{k_1!k_2!\cdots k_p!} \prod\limits_{j=1}^p (f_j)^{(k_j)},
\end{align*}
we deduce the desired evaluations \eqref{EQ-MPL-LOG-A} and \eqref{EQ-MAL-LOG-A} by direct calculations.
\end{proof}

For example, we have
\begin{align*}
\int_0^1 {\frac{{\rm Li}_{2,2}(x)\log(1-x)}{x(1-x)^{\alpha}}dx}
 =-2\zeta(3,2,1;1-\alpha)-2\zeta(2,3,1;1-\alpha)-\zeta(2,2,2;1-\alpha).
\end{align*}

Setting $1-e^{-t}=x$ and $s=k+1\in \N$ in \eqref{a1} and \eqref{a2}, we get
\begin{align}
&\xi(k+1;\bfk)=\frac{(-1)^{k}}{k!}\int_{0}^1 \frac{\log^k(1-x){\mathrm{Li}}_{\bfk}(x)}{x}dx,\label{AKZCV}\\
&\eta(k+1;\bfk)=\frac{(-1)^{k-1}}{k!}\int_{0}^1 \frac{\log^k(1-x){\mathrm{Li}}_{\bfk}
    (\frac{x}{x-1})}{x}dx.\label{KTECV}
\end{align}
Similarly, putting $x=\tanh(t/2)$ and $s=k+1\in \N$ in \eqref{defn-formula-KTP}, we get
\begin{align}\label{KTPsiCV}
\psi(k+1;\bfk)=\frac{(-1)^k}{k!}\int_{0}^1 \frac{\log^k(\frac{1-x}{1+x}){\rm A}(\bfk;x)}{x}dx.
\end{align}

\begin{cor}\label{AKZ-KTPSI-COR} For any composition $\bfk:=(k_1,k_2,\ldots,k_r)$ and any $k\in \N_0$, we have
\begin{align}
&\xi(k+1;\bfk)=\sum_{|\bfj|=k,\dep(\bfj)=n} B\left((\overleftarrow{\bfk}^\vee)_+;\bfj\right)
    \ze\left((\overleftarrow{\bfk}^\vee)_++\bfj\right),\label{eq-ef-AKZF}\\
&\psi(k+1;\bfk)=\sum_{|\bfj|=k,\dep(\bfj)=n} B\left((\overleftarrow{\bfk}^\vee)_+;\bfj\right)
T\left((\overleftarrow{\bfk}^\vee)_++\bfj\right),
\end{align}
where $n:=|\bfk|+1-\dep(\bfk)$.
\end{cor}
\begin{proof}
The corollary follows immediately from Theorem \ref{thm-MPL-MAL-Para} if we let $\alpha\rightarrow 0$.
\end{proof}

\begin{thm} \label{conj:starZeta}
For any composition $\bfk:=(k_1,k_2,\ldots,k_r)$ and any $\alpha\in \mathbb{C}\backslash \N$, the following formula holds:
\begin{align}\label{eq-exfl-KTeta}
\int_0^1 \frac{\Li_{\bfk}(\frac{x}{x-1})}{x(1-x)^\alpha}dx=(-1)^r \ze^\star \left((\overleftarrow{\bfk}^\vee)_+;1-\alpha\right).
\end{align}
\end{thm}

\begin{proof}
Let $\ta=d\log t$ and $\tb=-d\log(1-t)$. Then by the change of variable $t\to t/(t-1)$ we get
\begin{align*}
\Li_{\bfk} \left(\frac{x}{x-1} \right)=&\, \int_0^{\tfrac{x}{x-1}} \tb \ta^{k_r-1} \cdots  \tb \ta^{k_1-1} \\
=&\, (-1)^r \int_0^x \tb (\ta+\tb)^{k_r-1} \cdots  \tb (\ta+\tb)^{k_1-1}=(-1)^r \sum_{\bfl\succeq \bfk} \Li_{\bfl}(x),
\end{align*}
where $\bfl\succeq \bfk$ means that $\bfl$ runs through all compositions such that $\bfk$ can be obtained
from $\bfl$ by combining some of the consecutive parts of $\bfl$.  By \eqref{EQ-MPL-LOG-A} in
Theorem \ref{thm-MPL-MAL-Para} we have
\begin{equation*}
\int_0^1 \frac{\Li_{\bfk}(\frac{x}{x-1})}{x(1-x)^\alpha}dx
=(-1)^r \sum_{\bfl\succeq \bfk}
\ze\left((\overleftarrow{\bfl}^\vee)_+;1-\alpha\right).
\end{equation*}
Since taking duality and then adding 1 to the first components reverse the inclusion relations, we obtain
\begin{equation*}
 \sum_{\bfl\succeq \bfk}
\ze\left((\overleftarrow{\bfl}^\vee)_+;1-\alpha\right)
 =\sum_{\bfl'\preceq (\overleftarrow{\bfk}^\vee)_+} \ze \left( \bfl';1-\alpha\right)
=\ze^\star \left((\overleftarrow{\bfk}^\vee)_+;1-\alpha\right).
\end{equation*}
This completes the proof of the theorem.
\end{proof}

\begin{cor}\label{KTeta-COR} For any composition $\bfk:=(k_1,k_2,\ldots,k_r)$ and any $k\in \N_0$, we have
\begin{align}\label{eq-KTEF-exfor}
\eta(k+1;\bfk)&=(-1)^{r-1} \sum_{|\bfj|=k,\dep(\bfj)=n} B
\left((\overleftarrow{\bfk}^\vee)_+;\bfj\right)\ze^\star
\left((\overleftarrow{\bfk}^\vee)_++\bfj\right),
\end{align}
where $n:=|\bfk|+1-\dep(\bfk)$.
\end{cor}
\begin{proof}
Applying \eqref{KTECV} gives
\[\eta(k+1;\bfk)=-\frac{1}{k!}\lim_{\alpha\rightarrow 0}\frac{d^k}{d\alpha^k}\int_0^1 \frac{\Li_{\bfk}(\frac{x}{x-1})}{x(1-x)^\alpha}dx.\]
Then, using \eqref{eq-exfl-KTeta} we deduce
\begin{align*}
\eta(k+1;\bfk)&=\frac{(-1)^{r-1}}{k!}\lim_{\alpha\rightarrow 0} \frac{d^k}{d\alpha^k}\ze^\star \left((\overleftarrow{\bfk}^\vee)_+;1-\alpha\right)\\
&=(-1)^{r-1} \lim_{\alpha\rightarrow 0}\sum_{|\bfj|=k,\dep(\bfj)=n}
B\left((\overleftarrow{\bfk}^\vee)_+;\bfj\right)
\ze^\star\left((\overleftarrow{\bfk}^\vee)_++\bfj;1-\alpha\right)\\
&=(-1)^{r-1} \sum_{|\bfj|=k,\dep(\bfj)=n}
B\left((\overleftarrow{\bfk}^\vee)_+;\bfj\right)
\ze^\star\left((\overleftarrow{\bfk}^\vee)_++\bfj\right).
\end{align*}
 This concludes the proof of the corollary.
\end{proof}

We mention that \eqref{eq-ef-AKZF} and \eqref{eq-KTEF-exfor} were also proved in \cite{KawasakiOh2018}.

\section{A formula related to Kaneko--Tsumura's conjecture}\label{Sec.KT.conj}

Let $\mathcal{Z}$ be the space of usual multiple zeta values. In \cite[Conjecture 5.3]{KanekoTs2019}, the first-named author and Tsumura observed that the following holds:
\begin{equation}\label{KT.conj}
\sum_{\substack{i+j=m\\i,j\geq 0}}
    \binom{p+i-1}{i}\binom{q+j-1}{j}T(p+i,q+j)\in\mathcal{Z}\,,
\end{equation}
for $m,q\geq 1$ and $p\geq 2$, with $m+p+q$ even. That is, these sums are expressible in terms of MZVs. In 2021, Murakami \cite[Theorem 42]{Mura21} proved this conjecture by using the motivic method employed in \cite{Gla18} (see also \cite[Remark 5.6]{KanekoTs2019}).

In this section, we give a formula which involves a kind of symmetrized quantity of the left-hand side of~\eqref{KT.conj}.  The formula holds true for MZVs as well in the exact same form.  This suggests that there should be an algebraic proof using the shuffle product, but we give here our original analytic
proof using Theorem~\ref{thm-MPL-MAL-Para}.

\begin{thm}\label{thm-sym-for-KTCon} For any positive integers $q$ and $p,m\geq 2$, we have
\begin{align}
&\sum_{\substack{i+j=m-1\\i,j\geq 0}}
    \binom{p+i-1}{i}\binom{q+j-1}{j}\ze(p+i,q+j)\nonumber\\
&\quad-(-1)^q\sum_{\substack{i+j=p-1\\i,j\geq 0}}
    \binom{m+i-1}{i}\binom{q+j-1}{j}\ze(m+i,q+j)\nonumber\\
&\quad=\sum_{\substack{i+j=q-1\\i,j\geq 0}} (-1)^j
    \binom{m+i-1}{i}\binom{p+j-1}{j}\ze(m+i)\ze(p+j),\label{eq-sym-for-MZVCon}\\
&\sum_{\substack{i+j=m-1\\i,j\geq 0}}
    \binom{p+i-1}{i}\binom{q+j-1}{j}T(p+i,q+j)\nonumber\\
&\quad-(-1)^q\sum_{\substack{i+j=p-1\\i,j\geq 0}}
    \binom{m+i-1}{i}\binom{q+j-1}{j}T(m+i,q+j)\nonumber\\
&\quad=\sum_{\substack{i+j=q-1\\i,j\geq 0}} (-1)^j
    \binom{m+i-1}{i}\binom{p+j-1}{j}T(m+i)T(p+j).\label{eq-sym-for-KTCon}
\end{align}
\end{thm}
\begin{proof}
We first prove \eqref{eq-sym-for-KTCon}. Letting $(k_1,\ldots,k_r)=(m_p,\ldots,m_2,m_1-1)^\vee$ in \eqref{EQ-MAL-LOG-A} yields
\begin{align}\label{EQ-KTA-LOG-A-change}
&\int_0^1 \frac{\A((m_p,\ldots,m_2,m_1-1)^\vee;x)\log^k(\frac{1-x}{1+x})}{x(\frac{1-x}{1+x})^\alpha}dx\nonumber\\
&=(-1)^kk!\sum_{\substack{i_1+i_2+\cdots+i_p=k\\ i_j\geq 0,\ \forall j}}
\left\{\prod\limits_{j=1}^p\binom{m_j+i_j-1}{i_j}\right\}T(m_1+i_1,m_2+i_2,\ldots,m_p+i_p;1-\alpha).
\end{align}
Setting $p=2,\alpha=0$ in \eqref{EQ-KTA-LOG-A-change}, then replacing $(m_1,m_2,k)$ by $(p,q,m-1)$ and noting that $(q,p-1)^\vee=(\{1\}_{q-1},2,\{1\}_{p-2})$ ($\{1\}_r$ is the sequence of $1$'s with $r$ repetitions), we obtain
\begin{align}\label{KT-conj-A-int}
&\sum_{\substack{i+j=m-1\\i,j\geq 0}}
    \binom{p+i-1}{i}\binom{q+j-1}{j}T(p+i,q+j)\nonumber\\
&\quad=2\frac{(-1)^{m-1}}{(m-1)!} \int_0^1 \frac{\A(\{1\}_{q-1},2,\{1\}_{p-2};x)\log^{m-1}(\frac{1-x}{1+x})}{x}dx\nonumber\\
&\quad=2\frac{(-1)^{m-1}}{(m-1)!} \int_0^1 \frac{\A(\{1\}_{q-1},2,\{1\}_{p-2};\frac{1-t}{1+t})\log^{m-1}(t)}{1-t^2}dt,
\end{align}
where we have used the change of variable $x=\frac{1-t}{1+t}$ in the last step. From \cite[Thm 2.9]{PX2019}, we get
\begin{align}\label{KTA-change}
\A\left(\{1\}_{q-1},2,\{1\}_{p-2};\frac{1-t}{1+t}\right)&=(-1)^{q-1}\sum_{j=0}^{q-1}\frac{\binom{p+j-1}{j}}{(q-1-j)!}T(p+j)\log^{q-1-j}(t)\nonumber\\
&\quad+(-1)^{p+q-1}\sum_{j=0}^{p-1} (-1)^j \frac{\binom{q+j-1}{j}}{(p-1-j)!}\log^{p-1-j}(t)\A(q+j;t).
\end{align}
Hence, substituting \eqref{KTA-change} into \eqref{KT-conj-A-int} gives
\begin{align}\label{KT-conj-A-int-change}
&\sum_{\substack{i+j=m-1\\i,j\geq 0}}
    \binom{p+i-1}{i}\binom{q+j-1}{j}T(p+i,q+j)\nonumber\\
&\quad=2\frac{(-1)^{m+q}}{(m-1)!}\sum_{j=0}^{q-1}\frac{\binom{p+j-1}{j}}{(q-1-j)!}T(p+j)\int_0^1\frac{\log^{m+q-2-j}(t)}{1-t^2}dt\nonumber\\
&\quad\quad+2\frac{(-1)^{m+p+q}}{(m-1)!}\sum_{j=0}^{p-1} (-1)^j \frac{\binom{q+j-1}{j}}{(p-1-j)!}\int_0^1\frac{\log^{m+p-2-j}(t)\A(q+j;t)}{1-t^2}dt\nonumber\\
&\quad=\sum_{j=0}^{q-1} (-1)^j \binom{p+j-1}{j} \binom{m+q-2-j}{m-1}T(p+j)T(m+q-1-j)\nonumber\\
&\quad\quad+(-1)^q \sum_{j=0}^{p-1} \binom{q+j-1}{j}\binom{m+p-2-j}{m-1}T(m+p-1-j,q+j).
\end{align}
Then, a simple calculation yields \eqref{eq-sym-for-KTCon} easily.  The proof of \eqref{eq-sym-for-MZVCon} is similar to the above so we leave it to the interested reader.
\end{proof}

Moreover, letting $m=p$ in Theorem \ref{thm-sym-for-KTCon}, we can get the following corollary.

\begin{cor} For positive integers $q$ and $p\geq 2$, we have
\begin{align}
&(1-(-1)^q)\sum_{\substack{i+j=p-1\\i,j\geq 0}}
    \binom{p+i-1}{i}\binom{q+j-1}{j}\ze(p+i,q+j)\nonumber\\
&\quad=\sum_{\substack{i+j=q-1\\i,j\geq 0}} (-1)^j
    \binom{p+i-1}{i}\binom{p+j-1}{j}\ze(p+i)\ze(p+j),\label{eq-sym-for-KTCon-cor1}\\
&(1-(-1)^q)\sum_{\substack{i+j=p-1\\i,j\geq 0}}
    \binom{p+i-1}{i}\binom{q+j-1}{j}T(p+i,q+j)\nonumber\\
&\quad=\sum_{\substack{i+j=q-1\\i,j\geq 0}} (-1)^j
    \binom{p+i-1}{i}\binom{p+j-1}{j}T(p+i)T(p+j).\label{eq-sym-for-KTCon-cor2}
\end{align}
In particular, when $q$ is even, the two summations on the right side of (\ref{eq-sym-for-KTCon-cor1}) and (\ref{eq-sym-for-KTCon-cor2}) vanish.
\end{cor}

\medskip\noindent
{\bf Acknowledgments}

Masanobu Kaneko is supported by the JSPS KAKENHI Grant Numbers JP16H06336, JP21H04430. Weiping Wang is supported by the Zhejiang Provincial Natural Science Foundation of China (Grant No. LY22A010018). The corresponding author Ce Xu is supported by the National Natural Science Foundation of China (Grant No. 12101008), the Natural Science Foundation of Anhui Province (Grant No. 2108085QA01) and the University Natural Science Research Project of Anhui Province (Grant No. KJ2020A0057). Jianqiang Zhao is supported by the Jacobs Prize from The Bishop's School.

\medskip\noindent
{\bf Data availability statement} There is no data associated with the paper.

\medskip\noindent
{\bf Declarations of competing interest} none.

\end{document}